\documentclass[12pt,reqno]{article}

\usepackage[usenames]{color}
\usepackage{amssymb}
\usepackage{amsmath}
\usepackage{amsthm}
\usepackage{amsfonts}
\usepackage{amscd}
\usepackage{graphicx}
\usepackage{lscape}

\usepackage[colorlinks=true,
linkcolor=webgreen,
filecolor=webbrown,
citecolor=webgreen]{hyperref}

\definecolor{webgreen}{rgb}{0,.5,0}
\definecolor{webbrown}{rgb}{.6,0,0}

\usepackage{color}
\usepackage{fullpage}
\usepackage{float}

\usepackage{graphics}
\usepackage{latexsym}
\usepackage{epsf}
\usepackage{breakurl}

\newcommand{\seqnum}[1]{\href{https://oeis.org/#1}{\rm \underline{#1}}}
\def\suchthat{\, : \, }

\def\Enn{\mathbb{N}}
\def\Que{\mathbb{Q}}

\newenvironment{smallarray}[1]
{\null\,\vcenter\bgroup\scriptsize
\arraycolsep=.13885em
\hbox\bgroup$\array{@{}#1@{}}}
{\endarray$\egroup\egroup\,\null}

\begin{document}

\theoremstyle{plain}
\newtheorem{theorem}{Theorem}
\newtheorem{corollary}[theorem]{Corollary}
\newtheorem{lemma}[theorem]{Lemma}
\newtheorem{proposition}[theorem]{Proposition}

\theoremstyle{definition}
\newtheorem{definition}[theorem]{Definition}
\newtheorem{example}[theorem]{Example}
\newtheorem{conjecture}[theorem]{Conjecture}

\theoremstyle{remark}
\newtheorem{remark}[theorem]{Remark}

\title{Additive Properties of the Evil and Odious Numbers and Similar Sequences}

\author{
Jean-Paul Allouche \\
CNRS, IMJ-PRG, Sorbonne \\
4 Place Jussieu \\
F-75252 Paris Cedex 05\\
France \\
\href{mailto:jean-paul.allouche@imj-prg.fr}{\tt jean-paul.allouche@imj-prg.fr}
\and
Jeffrey Shallit\footnote{Research funded by a grant from NSERC, 2018-04118.} \\
School of Computer Science \\
University of Waterloo \\
Waterloo, Ontario  N2L 3G1\\ Canada \\
\href{mailto:shallit@uwaterloo.ca}{\tt shallit@uwaterloo.ca}
}

\maketitle

\begin{abstract}
First we reprove two results in additive number theory
due to Dombi and Chen \& Wang, respectively, on the number of
representations of $n$ as the sum of two odious or evil numbers, 
using techniques from automata theory and logic.
We also use this technique to prove a new result about
the numbers represented by five summands.

Furthermore, we prove some new results on the tenfold sums of the evil 
and odious numbers, as well as $k$-fold sums of similar sequences of
integers, by using techniques of analytic number theory involving 
trigonometric sums associated with the $\pm 1$ characteristic sequences 
of these integers.
\end{abstract}

\section{Introduction}

Let $\Enn = \{ 0,1,2,\ldots \}$ and let $A \subseteq \Enn$.
In a 1984 paper, Erd\H{o}s, S\'ark\"ozy, and S\'os \cite{Erdos&Sarkozy&Sos:1984} introduced three
functions based on $A$, as follows:\footnote{In fact, Erd\H{o}s, S\'ark\"ozy, and S\'os used
a different definition of $\Enn$ that excludes $0$.  It seems more natural to include $0$, and so (except in the last section) we adopt this convention.
One can easily get examples over the positive
integers by shifting the sets by $1$, which results in an ``off-by-$k$'' error when
taking sums of $k$ terms.}
\begin{align}
R_1^{(A)} (n) &= | \{ (x,y) \in \Enn \times \Enn  \suchthat x, y \in A \text{ and } x+y=n \} | 
\nonumber \\
R_2^{(A)} (n) &= | \{ (x,y) \in \Enn \times \Enn \suchthat x, y \in A \text{ and } x+y=n
	\text{ and } x<y \} | \label{eq1} \\
R_3^{(A)} (n) &= | \{ (x,y) \in \Enn \times \Enn \suchthat x, y \in A \text{ and } x+y=n
        \text{ and } x \leq y \} | . \nonumber 
\end{align}
Also see \cite{Sarkozy:2006}.

For $i \in \{1,2,3\}$,
apparently S\'ark\"ozy asked whether there exist two sets of
positive integers $A$ and $B$, with infinite symmetric difference,
for which $R_i^{(A)} (n) = R_i^{(B)} (n)$ for all sufficiently
large $n$.  A simple example of such sets was given by
Dombi \cite{Dombi:2002} in 2002, and we describe it next.
Actually, the same result had already appeared earlier in a paper of
Lambek and Moser \cite{Lambek&Moser:1959}.

Let ${\bf t} = t_0 t_1 t_2 \cdots$ be the Thue-Morse sequence,
defined by $t_0 = 0$, $t_{2n} = t_n$, and $t_{2n+1} = 1-t_n$
for $n \geq 0$.  It is easily seen that $t_n$ is the parity 
of the number of $1$'s (or sum of bits) in the binary
representation of $n$.
Let $A$ and $B$ be defined as follows:
\begin{align*}
A &= \{ n\geq 0\suchthat t_n = 0 \} = \{ 0,3,5,6,9,10,12,\ldots \}  \\
B &= \{ n\geq 0 \suchthat t_n = 1 \} = \{ 1,2,4,7,8,11,13, \ldots \}.
\end{align*}
These form a disjoint partition of $\Enn$.

In the literature, the set $A$ is sometimes called the
set of {\it evil numbers\/}, and the set $B$ is sometimes
called the set of {\it odious numbers}.
They are, respectively, sequences 
\seqnum{A001969} and \seqnum{A000069} in the 
{\it On-Line Encyclopedia of Integer Sequences} (OEIS) \cite{Sloane:2021}.
Dombi proved that $R_2^{(A)} (n) = R_2^{(B)} (n)$ for
$n \geq 0$.
His proof required $2{1\over 2}$ pages and a number of cases.
In Section~\ref{two} we show how to prove this using more-or-less routine
calculations involving finite automata and logic.

Chen \& Wang \cite{Chen&Wang:2003} proved a similar result
for the function $R_3$.   Instead of the Thue-Morse sequence,
they used a related sequence $t'_n$ counting the parity of
the number of $0$'s in the binary representation of $n$,
sometimes called the {\it twisted Thue-Morse sequence}.\footnote{However,
in some formulations, the twisted Thue-Morse sequence has $t'_0 = 0$.}
We have $t'_0 = 1$, $t'_1 = 0$, and $t'_{2n} = 1-t'_n$ and
$t'_{2n+1} = t'_n$ for $n \geq 1$. (Up to the first term it is
\seqnum{do} in the OEIS.)   

Chen \& Wang proved that if we set
\begin{align*}
C &= \{ n\geq 0 \suchthat t'_n = 0 \} = \{ 1,3,4,7,9,10,12,15, \ldots \}; \\
D &= \{ n\geq 0 \suchthat t'_n = 1 \} = \{ 0,2,5,6,8,11,13,14, \ldots \}.
\end{align*}
then $R_3^{(C)} (n) = R_3^{(D)} (n)$ for $n \geq 1$.\footnote{Again, there is 
an ``off-by-two'' difference in the way we stated the result, compared to 
the way they did.}  These sequences are, respectively
\seqnum{A059010} and \seqnum{A059009} in the OEIS.
Their proof required 3 pages and case analysis.
In this paper, in Section~\ref{three}, we reprove their results using techniques from automata
theory and logic.
For other proofs of the results of Dombi and Chen \& Wang, see
\cite{Sandor:2004,Lev:2004,Tang:2008}.

We can also consider generalizations of 
$R_1^{(A)} (n)$ to more than two summands, as follows:
\begin{align}
r_j (n) &:= |\{ (x_1, x_2, \ldots, x_j) \suchthat n = \sum_{1 \leq i \leq j} x_i 
	\text{ and } t_{x_i} = 0 \text{ for } 1 \leq i \leq j \} | \label{defr} \\
s_j (n) &:= |\{ (x_1, x_2, \ldots, x_j) \suchthat n = \sum_{1 \leq i \leq j} x_i 
	\text{ and } t_{x_i} = 1 \text{ for } 1 \leq i \leq j \} |, \label{defs}
\end{align}
where ${\bf t} = t_0 t_1 t_2 \cdots$ is the Thue-Morse sequence.
In Section~\ref{analysis-sec}, we prove a result from complex analysis that allows us to show that
both $r_{10} (n)$ and $s_{10}(n)$
are eventually strictly increasing functions of $n$.  By contrast, we can use our logical approach to show that
this is not the case for
$r_5 (n)$ and $s_5 (n)$.   The status for sums of
$6,7,8,$ and $9$ terms is currently unknown.  In Section~\ref{five} we prove some related results.

\section{Automata and first-order logic}
\label{two}

Our first proof technique depends on the fact that both $(t_n)$ and $(t'_n)$
are $k$-automatic sequences.  This means that, for each sequence, there exists
a deterministic finite automaton with output (DFAO) computing the sequence,
in the following sense:  when we feed the base-$k$ representation of $n$
into the automaton, it processes the digits and ends in a state $q$ with
output the $n$'th term of the sequence.  For these 
sequences we have $k = 2$.

For every $k$-automatic sequence $(a_n)$, there is a logical decision
procedure to decide the truth of assertions about the sequence that are 
phrased in the first-order logical structure $\langle \Enn, +, <, n\rightarrow a_n \rangle$.   
We call such a formula a {\it $k$-automatic formula}.  The results
are summarized in the following two theorems.
\begin{theorem}
Let $\varphi$ be a $k$-automatic formula.  There is a decision procedure
that, if $\varphi$ has no free variables, will either prove or disprove
$\varphi$.   Furthermore, if $\varphi$ has free variables $i_1, \ldots, i_k$,
then the procedure constructs a deterministic finite automaton
accepting the base-$k$ representation of those tuples $(i_1, \ldots, i_k)$
for which the formula evaluates to true.
\label{thm1}
\end{theorem}
For a proof, see \cite{Bruyere&Hansel&Michaux&Villemaire:1994}.

We now define the notion of linear representation of a function.
We say $f:\Enn\rightarrow\Que$ has a {\it linear representation of
rank $r$} if there exist an integer $k \geq 2$, a row vector $u \in \Que^r$,
a column vector $w \in \Que^r$, and an $r \times r$-matrix-valued
morphism $\gamma$ such that 
$f(n) = u \gamma(x) v$ for all
base-$k$ representations $x$ of $n$ (including those with leading zeros).

\begin{theorem}
There is an algorithm that, given a $k$-automatic formula
$\varphi$, with free variables
$i_1, i_2, \ldots, i_t, n$, computes 
a linear representation for $f(n)$, the number of
$t$-tuples of natural numbers $(i_1, i_2, \ldots, i_t)$
for which $\varphi(i_1, i_2, \ldots, i_t, n)$ is true.
\label{thm2}
\end{theorem}
For a proof, see \cite{Charlier&Rampersad&Shallit:2012}.

Finally, there is the notion of minimal linear representation, which
is a representation of smallest rank.    A well-known algorithm
of Sch\"utzenberger, based on linear algebra, takes a linear representation
and produces a minimal one from it \cite[\S 2.3]{Berstel&Reutenauer:2011}.

These are the basic tools we use to prove the results.  Theorems~\ref{thm1}
and \ref{thm2} have been implemented in free software called
{\tt Walnut}, originally created by Hamoon Mousavi \cite{Mousavi:2016,Shallit:2022}, and available at\\
\centerline{\url{https://cs.uwaterloo.ca/~shallit/walnut.html} \ .} 

\begin{theorem}
Suppose $(a_n)_{n \geq 0}$ is a $k$-automatic binary sequence and let
$A$ be the corresponding set $\{ n \suchthat a_n = 1 \}$.   
Then there is an algorithm producing the linear representation for
each of the functions $R_i^{(A)}(n)$, $i = 1,2,3$.
\end{theorem}

\begin{proof}
It suffices to give first-order logical formulas specifying that
$(x,y)$ is an ordered pair with sum $n$ corresponding to the pairs
in the definition \eqref{eq1}.
They are as follows:
\begin{align*}
R_1: & \quad  n=x+y \ \wedge\ a_x=1 \ \wedge\ a_y=1 \\ 
R_2: & \quad n=x+y \ \wedge\ x<y \ \wedge\ a_x=1 \ \wedge\ a_y=1 \\ 
R_3: & \quad n=x+y \ \wedge\ x\leq y \ \wedge\ a_x=1 \ \wedge\ a_y=1 
\end{align*}
Here, as usual, the symbol $\wedge$ denotes logical AND.
\end{proof}

We now give our proof of Dombi's result, which is based on routine
calculations using the results above.

\begin{theorem} (Dombi)
$R_2^{(A)} (n) = R_2^{(B)} (n)$ for $n \geq 0$.
\end{theorem}

\begin{proof}
The first step is to express the the set of pairs as a 
first-order formula.
We can do this as follows:
\begin{align*}
\varphi_A: & \  n=x+y \ \wedge\ x<y \ \wedge\ {\bf t}[x]=0 \ \wedge\ {\bf t}[y]=0 \\
\varphi_B: &\  n=x+y \ \wedge\ x<y \ \wedge\ {\bf t}[x]=1 \ \wedge\ {\bf t}[y]=1 .
\end{align*}

In {\tt Walnut} this is translated as
\begin{verbatim}
eval r2a "n=x+y & x<y & T[x]=@0 & T[y]=@0":
eval r2b "n=x+y & x<y & T[x]=@1 & T[y]=@1":
\end{verbatim}
The resulting automata, computed by {\tt Walnut}, both have 12 states.

Next, from these matrices we can immediately compute a
linear representation for the number of pairs $(x,y)$ making
the formula true.   To do so in {\tt Walnut} we use 
the following commands:
\begin{verbatim}
eval r2am n "n=x+y & x<y & T[x]=@0 & T[y]=@0":
eval r2bm n "n=x+y & x<y & T[x]=@1 & T[y]=@1":
\end{verbatim}
These commands create rank-$12$ linear representations for $R_2^{(A)} (n)$ and
$R_2^{(B)} (n)$, as follows:
$$
R_2^{(A)}(n) = (u, \gamma, v_A) \quad \text{ and } \quad
R_2^{(B)}(n) = (u, \gamma, v_B) ,$$
where
\begin{align*}
u &= \left[ \begin{smallarray}{c}
1\\
 0\\
 0\\
 0\\
 0\\
 0\\
 0\\
 0\\
 0\\
 0\\
 0\\
 0
\end{smallarray} \right]^T
& \gamma({\tt 0}) & = \left[ \begin{smallarray}{ccccccccccccc}
1&0&0&0&0&0&0&0&0&0&0&0\\
 0&0&0&1&1&0&0&0&0&0&0&0\\
 0&0&1&0&0&0&0&0&0&0&0&0\\
 0&0&0&0&0&0&0&0&1&1&1&0\\
 0&0&0&0&1&0&0&0&0&0&0&0\\
 1&0&0&0&0&0&0&0&0&0&0&1\\
 0&0&0&0&0&0&1&0&0&0&0&0\\
 0&0&0&0&0&0&0&1&0&0&0&0\\
 0&0&0&1&0&0&0&1&0&0&0&1\\
 0&0&0&1&0&0&1&0&0&0&0&1\\
 0&0&0&0&0&0&0&0&0&0&1&0\\
 0&0&1&0&0&0&0&0&1&1&0&0
\end{smallarray} \right]
&
\gamma({\tt 1}) &= \left[\begin{smallarray}{ccccccccccccc}
0&1&1&0&0&0&0&0&0&0&0&0\\
 0&0&0&0&0&1&0&0&0&0&0&0\\
 0&0&0&1&0&0&1&1&0&0&0&0\\
 0&0&0&0&0&0&0&0&0&0&0&1\\
 0&0&0&0&0&1&0&0&0&0&1&0\\
 0&1&0&0&0&0&0&0&0&0&0&0\\
 0&0&1&0&0&0&0&0&1&0&1&0\\
 0&0&1&0&0&0&0&0&0&1&1&0\\
 0&0&0&0&0&0&0&0&0&1&0&0\\
 0&0&0&0&0&0&0&0&1&0&0&0\\
 0&0&0&0&0&0&1&1&0&0&0&1\\
 0&0&0&1&0&0&0&0&0&0&0&0
\end{smallarray}\right]
&
v_A &= \left[ \begin{smallarray}{c}
0\\
 0\\
 0\\
 0\\
 0\\
 0\\
 0\\
 1\\
 0\\
 0\\
 0\\
 0
\end{smallarray} 
\right]  &
v_B &= \left[ \begin{smallarray}{c}
0\\
 0\\
 0\\
 0\\
 0\\
 0\\
 1\\
 0\\
 0\\
 0\\
 0\\
 0
 \end{smallarray}
\right] .
\end{align*}
Next, we apply the minimization algorithm to these two (slightly) different
linear representations, and discover that they both minimize to the same
linear representation $(u', \rho, v')$ of rank $5$, given as follows:
\begin{align*}
u' &= \left[ \begin{smallarray}{c}
1\\
 0\\
 0\\
 0\\
 0
\end{smallarray} \right]^T
& \rho({\tt 0}) & = \left[ \begin{smallarray}{rrrrr}
1& 0& 0& 0& 0\\
  0& 0& 1& 0& 0\\
  0& 0& 0& 0& 1\\
  0&-1& 0& 1& 1\\
 -2&-1& 3& 1& 0
\end{smallarray} \right]
&
\rho({\tt 1}) &= \left[\begin{smallarray}{rrrrr}
0& 1& 0& 0& 0\\
  0& 0& 0& 1& 0\\
 -1& 0& 1& 1& 0\\
 -2& 1& 1& 0& 1\\
 -1&-1& 0& 2& 1
\end{smallarray}\right]
&
v' &= \left[ \begin{smallarray}{c}
0\\
 0\\
 0\\
 1\\
 0
\end{smallarray} 
\right]  
\end{align*}
Since these two linear representations are the same, the result is proved.
\end{proof}

\begin{remark}
The sequence $R_2^{(A)} (n)$ is given as
sequence \seqnum{A133009} in the OEIS.
\end{remark}

One distinct advantage to this approach is that
a linear representation for $R_2^{(A)} (n)$ can be used to easily
prove additional results about it.  For example:
\begin{theorem}
For $t \geq 1$ we have
\begin{itemize}
\item[(a)] $R_2^{(A)}(2^t - 1) = 
\begin{cases}
0, & \text{if $t$ odd}; \\
2^{t-2}, & \text{if $t$ even};
\end{cases}$
\item[(b)] $R_2^{(A)}(2^t + 1) = \begin{cases}
(2^t +8)/12, & \text{if $t$ even}; \\
(2^t+4)/6, & \text{if $t$ odd}.
\end{cases}$
\end{itemize}
\end{theorem}

\begin{proof}
\begin{itemize}
\item[(a)]
Note that the base-$2$ representation of $2^t - 1$ consists of
the string $\overbrace{11\cdots 1}^t$.  Therefore
$$R_2^{(A)} (2^t - 1) = u' \rho(1)^t v'.$$
By well-known results,
the entries of $\rho(1)^t$ satisfy a linear
recurrence.  Therefore so does $u' \rho(1)^t v'$.
By the fundamental
theorem of linear recurrences, $u' \rho(1)^t v'$ can be expressed in terms of the
roots of the minimal polynomial of $\rho(1)$.  

This minimal polynomial is $X(X-1)(X-2)(X+2)$, and therefore
$R_2^{(A)} (2^t - 1) = A \cdot 2^t + B \cdot (-2)^t + C$ for
some constants $A, B, C$.  We can now solve for these constants
with the values of $R_2^{(A)} (2^t - 1)$ computed from the linear
representation to find that $A = 0$, $B = 1/8$, $C = 1/8$.
We therefore get
$R_2^{(A)} (2^t -1) =  2^{t-3} + (-2)^{t-3}$, which proves
the result.

\item[(b)]
We use the fact that $2^t + 1$ has base-$2$ representation
$1 \overbrace{00\cdots 0}^{t-1} 1$.  So it suffices to carry out
the same calculations as we did in part (a), except now they are
based on the minimal polynomial of $\rho(0)$.  It is the same
as for $\rho(1)$, namely  $X(X-1)(X-2)(X+2)$.
We then find (using the same technique as before) that
$R_2^{(A)} (2^t + 1) = 2/3 + 2^t/8 - (-2)^t/24$.
The result now follows.
\end{itemize}
\end{proof}

\section{Our proof of the Chen-Wang result}
\label{three}

\begin{theorem} (Chen-Wang)
With $C$ and $D$ defined as above, we have
$R_3^{(C)} (n) = R_3^{(D)} (n)$ for $n \geq 1$.
\end{theorem}

\begin{proof}
It is easy to see that
the sequence ${\bf t}' = t'_0 t'_1 t'_2 \cdots = 101001101 \cdots$
can be generated by the following DFAO:
\begin{figure}[H]
\begin{center}
\includegraphics[width=4in]{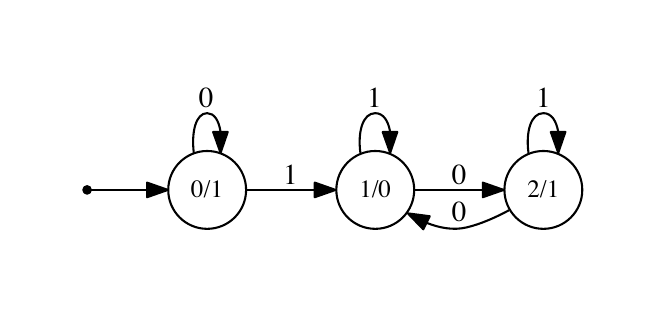}
\end{center}
\caption{DFAO computing $t'_n$}
\label{fig1}
\end{figure}
Here the labels of the states are given in the form
``state name/output of the state''.   

We start by translating the DFAO in Figure~\ref{fig1} into {\tt Walnut},
and store it as {\tt TT.txt} in {\tt Walnut's} {\tt Word Automata Library}.
\begin{verbatim}
msd_2
0 1
0 -> 0
1 -> 1

1 0
0 -> 2
1 -> 1

2 1
0 -> 1
1 -> 2
\end{verbatim}

We can then prove the equivalent result that
$R_3^{(C)} (n+1) = R_3^{(D)} (n+1)$ for $n \geq 0$.
\begin{verbatim}
eval r3cm n "n+1=x+y & x<=y & TT[x]=@0 & TT[y]=@0":
eval r3dm n "n+1=x+y & x<=y & TT[x]=@1 & TT[y]=@1":
\end{verbatim}

This gives us two linear representations, both of rank $20$.  When we minimize
these, as before, we get two identical minimized representations $(u,\gamma,v)$, as follows:
\begin{align*}
u &= \left[ \begin{smallarray}{c}
1\\
 0\\
 0\\
 0\\
 0\\
 0\\
 0\\
 0\\
 0\\
 0
\end{smallarray} \right]^T
& \gamma({\tt 0}) & = \left[ \begin{smallarray}{rrrrrrrrrrrr}
1& 0& 0& 0& 0& 0& 0& 0& 0& 0\\
  0& 0& 1& 0& 0& 0& 0& 0& 0& 0\\
  0& 0& 0& 0& 1& 0& 0& 0& 0& 0\\
  0& 0& 0& 0& 0& 0& 1& 0& 0& 0\\
  0& 0& 0& 0& 0& 0& 0& 0& 1& 0\\
  0& 0& 2&-2& 0&-1&-1& 2& 0& 1\\
  0&-2& 0& 0& 2& 1& 0& 0&-1& 1\\
  0&-1&-2& 1& 1& 1& 1&-1& 1& 0\\
  0& 0& 1&-2& 0&-1&-1& 2& 1& 1\\
  0&-3& 5&-3& 0& 0&-2& 3&-2& 3
\end{smallarray} \right]
&
\gamma({\tt 1}) &= \left[\begin{smallarray}{rrrrrrrrrr}
0& 1& 0& 0& 0& 0& 0& 0& 0& 0\\
  0& 0& 0& 1& 0& 0& 0& 0& 0& 0\\
  0& 0& 0& 0& 0& 1& 0& 0& 0& 0\\
  0& 0& 0& 0& 0& 0& 0& 1& 0& 0\\
  0& 0& 0& 0& 0& 0& 0& 0& 0& 1\\
  0&-2& 2& 0& 0& 1&-2& 1& 0& 1\\
  0&-1& 3&-2&-1& 0&-2& 3& 0& 1\\
  0& 0& 2&-2&-1&-1&-1& 2& 0& 2\\
  0& 0&-1&-1& 0& 1& 0& 0& 1& 1\\
  0&-3& 3&-1& 0& 0&-3& 4& 0& 1
\end{smallarray}\right]
&
v &= \left[ \begin{smallarray}{c}
0\\
 1\\
 0\\
 1\\
 1\\
 1\\
 1\\
 2\\
 0\\
 2
\end{smallarray} 
\right]  .
\end{align*}
and so the result is proved.
\end{proof}

\begin{remark}
The sequence $R_3^{(C)} (n)$ is sequence
\seqnum{A059451} in the OEIS.
\end{remark}

\section{Results for five and ten summands}
\label{analysis-sec}

In this section we show that the sequences $r_{10}$ and $s_{10}$, defined
above in Eqs.~\eqref{defr} and \eqref{defs}, are eventually strictly 
increasing. By contrast, as we will see later, the sequences $r_5$ and 
$s_5$ are not.  For $r_{10}$ and $s_{10}$, the ``logical approach'' of 
previous sections does not seem to suffice to prove the strictly 
increasing property, so we turn instead to techniques of analytic number
theory.

\medskip

Let ${\mathbf q} = (q_n)_{n \geq 0}$ be a sequence of $\pm 1$'s taking 
the value $+1$ infinitely often. For complex numbers $z$ and integers 
$n \geq 0$, we define the sums 
\begin{align*}
Q_n(z) &:= \sum_{0 \leq j \leq n} q_j z^j \\
Q(z) &:= \sum_{j \geq 0} q_j z^j \quad \text{(for $|z| < 1$)}. 
\end{align*}
We also define $L = L_{\bf q}$ by 
$L = L_{\bf q} := \{n \geq 1 \suchthat q_{n-1} = 1\}$ and 
$g(L,z) := \sum_{a \in L} z^a$. Let $r(k,L,n)$ denote the number of 
solutions of the equation 
$n = x_1 + \dots + x_k$ with $x_j \in L$ for all $j$. 

\begin{remark}
Note that, with the notation above, we have that $0 \notin L$.
To see that this does not restrict the generality, note that, if we want 
to represent the integers with $k$ summands, then, adding $1$ to every 
element of the underlying set just shifts the representation function 
by the additive constant $k$.
\end{remark}

\begin{theorem}
Suppose there exists a constant $C > 0$ and a real exponent 
$\alpha \in (0,1)$ such that, for all $z \in {\mathbb C}$ with $|z| = 1$ 
and for all $n \geq 1$, one has $|Q_n(z)| \leq C n^{\alpha}$. 
Then the sequence $r(k,L,n)$ is eventually strictly increasing for every 
integer $k$ such that $k > 2/(1-\alpha)$.
\label{analytic}
\end{theorem}
 
\begin{proof}
First, we note that the maximum modulus principle implies that 
$|Q_n(z)| \leq C n^{\alpha}$ for all $z$ with $|z| \leq 1$ and 
all $n \geq 1$.
We clearly have $g(L,z)^k = \sum_{n \in {\mathbb N}} r(k,L,n) z^n$. Since 
$$\Delta_{k,n} := r(k,L,n) - r(k,L,n-1)$$ is the coefficient of $z^n$ in the 
power series expansion of $(1-z) g(L,z)^k$, it suffices to prove that this 
coefficient is positive for $n > n_0(k)$. It thus suffices to prove that 
$\Delta_{k,n} > 0$ when $n$ 
is large enough. But, by Cauchy's differentiation formula, $\Delta_{k,n}$ 
is also equal to
$$
\Delta_{k,n} = 
\frac{1}{2i\pi} \oint_{\Gamma} (1-z) g(L,z)^k \frac{dz}{z^{n+1}}
$$
where $\Gamma$ is a (small) circle centered at the origin. 
Thus, taking for this circle of integration
$\Gamma = \Gamma_{k,n}:= 
\{z \suchthat z = re^{2i\pi t}, \ r = e^{-1/(n-k)}\}$, we have
\begin{equation}
\label{Delta}
\Delta_{k,n} = \int_0^1(1-z)g(L,z)^k z^{-n} dt, \ \text{with} \ z = r e^{2i\pi t} \ \text{and} \ r = e^{-1/(n-k)}.
\end{equation}
Since 
$$
g(L,z) = \sum_{a \in L} z^a = \sum_{j \geq 1} \frac{1}{2} (q_{j-1}+1) z^j = 
\frac{z}{2} \left(\frac{1}{1-z} + Q(z)\right),
$$
we obtain
\begin{equation}
\Delta_{k,n} = \int_0^1(1-z) \left(\frac{z}{2}\right)^k 
\left(\frac{1}{1-z} + Q(z)\right)^k z^{-n} dt 
= \int_0^1(1-z) \left(\frac{z}{2}\right)^k \left(\frac{1}{1-z} + 
Q_n(z)\right)^k z^{-n} dt .
\end{equation}
Note that
the terms in $Q$ corresponding to indices $> n$ give integrals equal 
to $0$. 

Hence
$$
\Delta_{k,n} = 2^{-k} \int_0^1 z^{-(n-k)}(1-z) \left(
\sum_{0 \leq \ell \leq k} {k \choose \ell} \frac{1}{(1-z)^{k-\ell}} 
Q_n^{\ell}(z)\right) dt.
$$ 
Now we split $\Delta_{k,n}$ into three quantities: the term 
corresponding to $\ell = 0$, the term $\ell = k$, and the term 
corresponding to $\ell \in [1, k-1]$.

\medskip

For $\ell = 0$ the corresponding term is 
$$
\begin{array}{lll}
2^{-k} \displaystyle\int_0^1 \frac{1}{(1-z)^{k-1}} z^{-(n-k)} dt &=&
2^{-k} \displaystyle\int_0^1 \left(\sum_{r \geq 0}  {k+r-2 \choose r} 
z^r\right) z^{-(n-k)} dt \\
&=& \displaystyle 2^{-k} {n-2 \choose n-k} = 
\displaystyle 2^{-k} {n-2 \choose k-2} \\
&\sim& \displaystyle 2^{-k} \frac{n^{k-2}}{(k-2)!}  \cdot
\end{array}
$$

For $\ell = k$, we use the upper bound $|Q_n(z)| \leq C n^{\alpha}$, 
thus obtaining the bound
$$
\left| 2^{-k} \int_0^1 (1-z) z^{-(n-k)} Q_n^k(z) dt \right| \leq 
2^{1-k} C^k e \ n^{k \alpha} .
$$

Now we look at the terms 
$$ 
I_{\ell} := 2^{-k} \int_0^1 \frac{1}{(1-z)^{k-\ell-1}} 
{k \choose \ell} z^{-(n-k)} Q_n^{\ell} dt
$$
for $\ell \in [1,k-1]$. Using the bound $| Q_n(z) | \leq C n^{\alpha}$ 
and the fact that $|z| = e^{-1/(n-k)}$, we obtain
\begin{equation}
 | I_{\ell} | \leq 2^{-k}{k \choose \ell} C^{\ell} n^{\alpha \ell} e
\int_0^1 \left| \frac{1}{(1-z)^{k-\ell-1}} \right| dt.  \label{star}
\end{equation}

Now, in order to evaluate the integral in (\ref{star}), we first note 
that (recall that $z = r e^{2i\pi t}$)
$$
\int_0^1 \left| \frac{1}{(1-z)^{k-\ell-1}} \right| dt =
2 \int_0^{1/2} \left| \frac{1}{(1-z)^{k-\ell-1}} \right| dt.
$$
Then, mimicking Dombi's method in \cite{Dombi:2002}, we split the 
interval $[0,1/2]$ into $[0,1/2] = J_1 \cup J_2$ where we define
$$
J_1 := [0, n^{-(\alpha+\varepsilon)}]\cup 
[1/2 - n^{-(\alpha+\varepsilon)},1/2] \ : \text{and} \ \
J_2 := [n^{-(\alpha+\varepsilon)}, 1/2 -n^{-(\alpha+\varepsilon)}],
$$
so that
$$
\int_0^{1/2} \left| \frac{1}{(1-z)^{k-\ell-1}} \right| dt =
\int_{J_1} \left| \frac{1}{(1-z)^{k-\ell-1}} \right| dt + 
\int_{J_2} \left| \frac{1}{(1-z)^{k-\ell-1}} \right| dt .
$$

\medskip

For $J_1$, since $|z| = r < 1$, we have when $n$ goes to infinity 
(recall that $k$ is fixed), that
$$
\left|\frac{1}{1-z}\right| \leq \frac{1}{1 - |z|} 
= \frac{1}{1 - e^{-1/(n-k)}} \sim n - k \sim n.
$$
Thus
$$
\left|\frac{1}{1-z}\right|^{k - \ell -1} = {\mathcal O}(n^{k - \ell -1}) \
\ \text{and} \ \
\int_{J_1} \left| \frac{1}{(1-z)^{k-\ell-1}} \right| dt = 
{\mathcal O}(n^{k - \ell - 1 - \alpha - \varepsilon}).
$$

\medskip

For $J_2$, we note that, for $x \in [\theta, \pi - \theta]$ (with 
$\theta \in (0, \pi/2)$), we have 
$\sin x \geq \sin \theta \geq (2/\pi) \theta$. Hence,
for $t \in J_2$ and $n$ large enough
$$
\left|\frac{1}{1-z}\right| \leq \left|\frac{1}{\Im(1-z)}\right| = 
\left|\frac{1}{r \sin (2 \pi t)}\right|
= {\mathcal O}(e^{1/(n-k)} n^{\alpha + \varepsilon}) = 
{\mathcal O}(n^{\alpha+\varepsilon}).
$$
Thus
$$
\left|\frac{1}{1-z}\right|^{k - \ell -1} = 
{\mathcal O}(n^{(\alpha + \varepsilon)(k - \ell -1)}) \ \ \text{and} \ \
\int_{J_2} \left| \frac{1}{(1-z)^{k-\ell-1}} \right| dt = 
{\mathcal O}(n^{(\alpha + \varepsilon)(k - \ell - 1)}).
$$
 
Finally, we obtain
$$
| I_{\ell} | = 
{\mathcal O}(n^{\alpha \ell + k - \ell - 1 - \alpha - \varepsilon})
+  {\mathcal O}(n^{\alpha \ell + (\alpha + \varepsilon)(k - \ell - 1)}).
$$
If $\displaystyle\alpha < \frac{k-2}{k-1}$, i.e., 
$\displaystyle k > \frac{2-\alpha}{1-\alpha}$, we can choose
$\displaystyle\varepsilon := \frac{k-2}{k-1} - \alpha > 0$. 
It is easy to check that this implies
$$
| I_{\ell} | = {\mathcal O}(n^{k-2-\varepsilon}) \ \ 
\text{for $\ell \in [1, k-1]$}:
$$
namely $\alpha(\ell-1) \leq (\ell-1)$, hence 
$\alpha \ell - \ell - \alpha \leq -1$, which gives
$\alpha \ell + k - \ell - 1 - \alpha - \varepsilon \leq k-2-\varepsilon$,
and $\alpha \ell+(\alpha+\varepsilon)(k-\ell-1) =  
((k-2)/(k-1)-\varepsilon)\ell + (k-\ell-1)(k-2)/(k-1) =
k-2 - \varepsilon \ell \leq k-2-\varepsilon$.

\medskip

Gathering the bounds for $| I_{k} |$ and $| I_{\ell} |$ for 
$\ell \in [1, k-1]$ we have
$$
\sum_{1 \leq \ell \leq k} | I_{\ell} | = {\mathcal O}(n^{k \alpha}) 
+ {\mathcal O}(n^{k-2-\varepsilon}) \ \
\text{provided that} \ \ k > \frac{2-\alpha}{1-\alpha}.
$$ 
Hence $\Delta_{k,n} \sim I_0 \sim 2^{-k} \frac{n^{k-2}}{(k-2)!}$ 
provided that $k > \frac{2-\alpha}{1-\alpha}$ and $k \alpha < k-2$. 
Since the condition $k \alpha < k-2$, i.e., the inequality 
$k(1-\alpha)  > 2$, implies that $k > \frac{2-\alpha}{1-\alpha}$, 
we are done. 
\end{proof}

\begin{corollary}
The sequences $r_{10}$ and $s_{10}$ are eventually strictly increasing.
\end{corollary}

\begin{proof}
We apply Theorem~\ref{analytic} to $r_{10}$ and $s_{10}$.  In this case
we take $q_n = (-1)^{t_n}$, and use the known fact 
\cite{Gelfond:1967,Newman&Slater:1975} that for this sequence we have
$\sup_{|z|=1} |Q_n(z)| \leq C n^{\alpha}$ for
$\alpha = (\log 3)/(\log 4) \doteq 0.79248$.
Since $10 > 2/(1-\alpha) \doteq 9.63768$,
we get that $r_{10}$ and $s_{10}$ are (eventually) strictly increasing 
functions of $n$.
\end{proof}

\bigskip

The status for $6,7,8,$ and $9$ summands is currently unknown.  Based on numerical evidence,
we make the following conjectures:
\begin{conjecture}
\leavevmode
\begin{itemize}
\item[(a)] Both $r_{6} (n)$ and $s_{6} (n)$ are
eventually strictly increasing.
\item[(a)] $r_6(n) < r_6(n+1)$ for $n \geq 37$.
\item[(b)] $s_6(n) < s_6(n+1)$ for $n \geq 5$.
\end{itemize}
\end{conjecture}

Now we turn our attention to $r_5$ and $s_5$.  In contrast to the 
situation for $r_{10}$ and $s_{10}$, we can use our ``logical approach'' to 
show that these sequences are {\it not\/} strictly increasing.

For any fixed $j$, one can easily obtain linear representations for 
$r_j$ and $s_j$ using the methods explained above.

\begin{theorem}
We have $r_5(2^n) > r_5(2^n + 1)$ and
$s_5(2^n) > s_5 (2^n + 1)$ for all sufficiently large $n$.
\end{theorem}

\begin{proof}
We can use {\tt Walnut} to compute a linear representation for
$r_5(n)$, as follows:
\begin{verbatim}
eval r5 n "n=i+j+k+l+m & T[i]=@0 & T[j]=@0 & T[k]=@0 & T[l]=@0 & T[m]=@0":
\end{verbatim}

This gives us vectors $v, w$ and a matrix-valued
morphism $\gamma$ such that $v \gamma(x) w = r_5(n)$ for all
binary strings $x$ such that $[x]_2 = n$.   The rank of this linear 
representation is $160$, and is not given here for space reasons.

Next, we compute the minimal polynomial of $\gamma(0)$ using
{\tt Maple}.  It is
$$X^4(X-1)(X-2)(X-4)(X-8)(X-16)(X+2)(X+4)(X+8)(X^2-8)(X^2-2X-16).$$
It follows that both $r_5(2^n)$ and $r_5(2^n+1)$ can be written as
a linear combination of the $n$'th powers of the zeros of this polynomial,
and therefore, so is the difference $r_5(2^n)-r_5(2^n + 1)$.
When we solve for the coefficients of this linear combination, we find that 
the coefficient corresponding to $16^n$ is positive (in fact it is
$1/14039101440$).  Since $16$ is the dominant root, this shows
the existence of some $n_0$ such that
the difference $r_5(2^n)-r_5(2^n+1)$ is positive for all $n \geq n_0$.

Exactly the same proof, word-for-word, works for $s_5$.
\end{proof}

\begin{remark}
Theorem~\ref{analytic} can be applied to several other sequences for 
which the condition $|Q_n(z)| \leq C n^{\alpha}$ for some $\alpha$ in 
$(0,1)$ holds. We give but one family of examples---namely, the 
Golay-Shapiro-Rudin sequences, for which it is known that $\alpha = 1/2$, and
hence $k > 4$.  For the usual Golay-Shapiro-Rudin sequence, this is exactly 
the first part of Dombi's Theorem~1 \cite[p.~138]{Dombi:2002}; more 
generally this also gives $k > 4$ for the generalized Rudin-Shapiro 
sequences of Theorem~3.1 in \cite[p.~20]{Allouche&Liardet:1991}, with 
$\varphi$ and $v$ being the constant sequence $1$.
\end{remark}

\section{Other results}
\label{five}

For the following result and proof, we adopt the
Iverson notation where, for a proposition $P$,
we set $[P] = 1$ if $P$ is true
and $[P] = 0$ otherwise.

\begin{theorem}
For $n \geq 0$ we have
$r_2 (n) - s_2 (n) = [$n$ {\rm\ even}] (-1)^{t_n}$.
\end{theorem}

\begin{proof}
We can find linear representations
for $r_2$ and $s_2$ with
the {\tt Walnut} commands 
\begin{verbatim}
eval r2m n "n=x+y & T[x]=@0 & T[y]=@0":
eval s2m n "n=x+y & T[x]=@1 & T[y]=@1":
\end{verbatim}
They are $(u_2, \gamma_2, v_2)$ for $r_2$ and
$(u_2, \gamma_2, v'_2)$ for $s_2$, where
$$
u_2 = \left[ \begin{smallarray}{c}
1\\
 0\\
 0\\
 0\\
 0\\
 0\\
 0\\
 0
\end{smallarray} \right]^T  \quad
\gamma_2(0) =
\left[ \begin{smallarray}{rrrrrrrr}
1&0&0&0&0&0&0&0\\
 0&0&0&0&1&1&1&0\\
 0&0&1&0&0&0&0&0\\
 0&0&0&1&0&0&0&0\\
 0&1&0&1&0&0&0&1\\
 0&1&1&0&0&0&0&1\\
 0&0&0&0&0&0&1&0\\
 1&0&0&0&1&1&0&0\\
\end{smallarray} \right] \quad
\gamma_2(1) =
\left[ \begin{smallarray}{rrrrrrrr}
0&1&1&1&0&0&0&0\\
 0&0&0&0&0&0&0&1\\
 1&0&0&0&1&0&1&0\\
 1&0&0&0&0&1&1&0\\
 0&0&0&0&0&1&0&0\\
 0&0&0&0&1&0&0&0\\
 0&0&1&1&0&0&0&1\\
 0&1&0&0&0&0&0&0
\end{smallarray} \right] \quad
v_2 = \left[ \begin{smallarray}{c}
1\\
 0\\
 0\\
 0\\
 0\\
 0\\
 0\\
 0
\end{smallarray}\right] \quad
v'_2 = \left[ \begin{smallarray}{c}
 0\\
 0\\
 0\\
 0\\
 0\\
 0\\
 1\\
 0
\end{smallarray}\right].
$$
From this we can easily form a linear representation 
for $r_2(n) - s_2 (n)$ as follows:
$(u_2, \gamma_2, v_2-v'_2)$.  When we minimize it, 
we get a linear representation $(x_2, \gamma'_2, y_2)$ of rank 2, as follows:
$$
x_2 = \left[ \begin{smallarray}{c}
1\\
 0\\
\end{smallarray} \right]^T  \quad
\gamma'_2(0) =
\left[ \begin{smallarray}{rr}
1 & 0 \\
-1 & 0
\end{smallarray} \right] \quad
\gamma'_2(1) =
\left[ \begin{smallarray}{rr}
0 & 1 \\
0 & -1 \\
\end{smallarray} \right] \quad
y_2 = \left[ \begin{smallarray}{c}
1\\
 0\\
\end{smallarray}\right].
$$
Now an easy induction gives that 
$$\gamma_2'(x) = 
\left[ \begin{array}{rr}
[n {\rm\ even}](-1)^{t_n} & -[n {\rm\ odd}] (-1)^{t_n} \\
-[n {\rm\ even}] (-1)^{t_n} & [n {\rm\ odd}] (-1)^{t_n}
\end{array}
\right] $$
for $n \geq 1$ and all strings $x$ such that $[x]_2 = n$.
This completes the proof.
\end{proof}

\begin{theorem}
There are infinitely many $n$ for which $r_3 (n) = s_3 (n)$.
Some examples include $n = 4^i - 2$ for $i \geq 1$ and
$n = 3\cdot 4^i - 1$ for $i \geq 0$.
\end{theorem}

\begin{proof}
We can find linear representations for $r_3(n)$ and $s_3(n)$
using the following {\tt Walnut} commands:
\begin{verbatim}
eval r3m n "n=x+y+z & T[x]=@0 & T[y]=@0 & T[z]=@0":
eval s3m n "n=x+y+z & T[x]=@1 & T[y]=@1 & T[z]=@1":
\end{verbatim}
It turns out these linear representations are of rank $24$ and of the form
$(u_3, \gamma_3, v_3)$ and $(u_3, \gamma_3, v'_3)$,
respectively.    So we can form the linear
representation for $r_3(n)-s_3(n)$ by
$(u_3, \gamma_3, v_3-v'_3)$.  When we minimize it,
we get a linear representation $(x_3, \gamma'_3, y_3)$ of
rank $6$, as follows:
$$
x_3 = \left[ \begin{smallarray}{c}
1\\
 0\\
 0\\
 0\\
 0\\
 0
\end{smallarray} \right]^T  \quad
\gamma'_3(0) =
\left[ \begin{smallarray}{rrrrrr}
1& 0& 0& 0& 0& 0\\
  0& 0& 1& 0& 0& 0\\
  0& 0& 0& 0& 1& 0\\
 -1&-1&-1& 0&-1&-1\\
 -3& 1& 0& 0&-2& 1\\
  2& 5&-3& 2& 4& 3
\end{smallarray} \right] \quad
\gamma'_3(1) =
\left[ \begin{smallarray}{rrrrrr}
0& 1& 0& 0& 0& 0\\
  0& 0& 0& 1& 0& 0\\
  0& 0& 0& 0& 0& 1\\
  3&-3& 2&-3& 1&-1\\
 -3&-1&-2& 2&-1&-2\\
 -1& 2& 0&-1&-1& 2
\end{smallarray} \right] \quad
y_3 = \left[ \begin{smallarray}{c}
1\\
  0\\
  0\\
  2\\
 -3\\
  0
\end{smallarray}\right] .
$$
Now the binary representation of $4^i - 2$ is of the form
$1^{2i-1} 0$, so we know that $r(4^i -2)-s(4^i - 2)$
can be expressed as a linear combination of the $(2i-1)$th powers of
the roots of the
minimal polynomial of $\gamma'_3(1)$.
This minimal polynomial is
$X^2 (X+1)(X^2 - 8)$.    Solving for this linear combination,
we find that the coefficients are all zero, so 
$r(4^i -2)-s(4^i - 2) = 0$ for all $i \geq 1$.  Actually, with the
same technique, one can prove that
$r(4^i-2)=s(4^i-2) = 16^{i-1} - 4^{i-1}$.

For $3 \cdot 4^i - 1$, the same ideas work.
\end{proof}

\begin{theorem}
There are infinitely many $n$ for which $r_4 (n) = s_4 (n)$.
Some examples include $n = 6 \cdot 4^i - 1$ for $i \geq 0$ and
$n = 2\cdot 4^i - 3$ for $i \geq 1$.
\end{theorem}

\begin{proof}
We can find linear representations for $r_4(n)$ and $s_4(n)$
using the following {\tt Walnut} commands:
\begin{verbatim}
eval r4m n "n=x+y+z+w & T[x]=@0 & T[y]=@0 & T[z]=@0 & T[w]=@0":
eval s4m n "n=x+y+z+w & T[x]=@1 & T[y]=@1 & T[z]=@1 & T[w]=@1":
\end{verbatim}
It turns out these linear representations are of rank $64$ and of the form
$(u_4, \gamma_4, v_4)$ and $(u_4, \gamma_4, v'_4)$,
respectively.    So we can form the linear
representation for $r_4(n)-s_4(n)$ by
$(u_4, \gamma_4, v_4-v'_4)$.  When we minimize it,
we get a linear representation $(x_4, \gamma'_4, y_4)$ of rank $7$,
as follows:
$$
x_4 = \left[ \begin{smallarray}{c} 
 1\\
 0\\
 0\\
 0\\
 0\\
 0\\
 0
\end{smallarray} \right]^T  \ 
\gamma'_4(0) =
\left[ \begin{smallarray}{rrrrrrr}
    1&   0&   0&   0&   0&   0&   0\\
    0&   0&   1&   0&   0&   0&   0\\
    0&   0&   0&   0&   1&   0&   0\\
    0&   0&   0&   0&   0&   0&   1\\
   -3&   1&   0&   0&  -2&   1&   0\\[1pt]
   -{1 \over 2}&   {5\over 2}&  -{{11}\over 2}&   2&   {3 \over 2}&   {1 \over 2}&  -{5 \over 2}\\[3pt]
    {5\over 2}&  -{7 \over 2}&   {9 \over 2}&  -2&  -{1\over 2}&  -{3 \over 2}&  {5\over 2} 
\end{smallarray} \right] \ 
\gamma'_4(1) =
\left[ \begin{smallarray}{rrrrrrr}
    0&   1&   0&   0&   0&   0&   0\\
    0&   0&   0&   1&   0&   0&   0\\
    0&   0&   0&   0&   0&   1&   0\\[1pt]
    {7 \over 2}&  -{5\over 2}&   {5\over 2}&  -3&   {3 \over 2}&  -{1\over 2}&   {1 \over 2}\\[3pt]
   -{5 \over 2}&  -{1\over 2}&  -{3\over 2}&   2&  -{1 \over 2}&  -{3 \over 2}&   {1\over 2}\\[3pt]
   -1&   2&   0&  -1&  -1&   2&   0\\[1pt]
    {7\over 2}&  -{5\over 2}&   {3\over 2}&  -2&   {3\over 2}&  -{3\over 2}&  -{1\over 2}
\end{smallarray} \right] \ 
y_4 = \left[ \begin{smallarray}{c}
1\\
  0\\
  0\\
  4\\
 -1\\
  0\\
  4
\end{smallarray}\right] .
$$
Now the binary representation of $6 \cdot 4^i - 1$ is of the form
$10 1^{2i+1}$, so we know that $r_4(6 \cdot 4^i -1 )-s_4(6\cdot 4^i - 1)$
can be expressed as a linear combination of the $(2i+1)$th powers of
the roots of the
minimal polynomial of $\gamma'_4(1)$.
This minimal polynomial is
$X^3 (X+1)(X^2 - 8)$.    Solving for this linear combination,
we find that the coefficients are all zero, so 
$r_4(6 \cdot 4^i -1 )-s_4(6\cdot 4^i - 1) = 0$
for all $i \geq 0$.  In fact, with a little more work, and the same technique,
one can show that 
$$r_4(6 \cdot 4^i -1 )=s_4(6\cdot 4^i - 1)=
{9 \over 4} 64^i + 16^i + {{4^i} \over 8} + c_1 \alpha_1^i + c_2 \alpha_2^i,$$
where $\alpha_1 = 18 - 2\sqrt{17}$, $\alpha_2 = 18+2 \sqrt{17}$,
$c_1 = (7\alpha_1-2\alpha_2)/288$, $c_2 = (7\alpha_2 - 2\alpha_1)/288$.

For $2\cdot 4^i - 3$, the same technique works.
\end{proof}

\section*{Acknowledgments}

We thank Emmanuel Lesigne for raising the problem and Michel Dekking for
discussions.

\end{document}